\documentclass[a4paper,reqno,11pt]{amsart}
\usepackage[utf8]{inputenc}
\usepackage[margin=0.8in]{geometry}
\usepackage[svgnames]{xcolor}
\usepackage[bookmarksnumbered=true]{hyperref} 
\hypersetup{
	colorlinks = true,
	linkcolor = blue,
	anchorcolor = blue,
	citecolor = teal,
	filecolor = blue,
	urlcolor = blue
}
\usepackage[english]{babel}
\usepackage{amsmath}
\usepackage{amsthm}
\usepackage{amssymb}
\usepackage{mathrsfs}
\usepackage{tikz-cd} 
\usepackage{bm}

\usepackage{hyperref}
\usepackage{caption}

\usepackage{cleveref}

\title{Voronoi diagrams of algebraic varieties under polyhedral norms}
\author[A. Becedas]{Adrian Becedas}
\address{Department of Mathematics, KTH Royal Institute of Technology, Stockholm, Sweden}
\email{adrian.becedas@gmail.com}
\author[K. Kohn]{Kathl\'en Kohn}
\address{Department of Mathematics, KTH Royal Institute of Technology, Stockholm, Sweden}
\email{kathlen@kth.se}
\author[L. Venturello]{Lorenzo Venturello}
\address{Department of Mathematics, KTH Royal Institute of Technology, Stockholm, Sweden, and \\ Dipartimento di Matematica, Universit\`{a} di Pisa, Pisa, Italy}
\email{lven@kth.se, lorenzo.venturello@unipi.it}
\date{September 2022}

\usepackage{graphicx,stackengine}
\usepackage{float}

\newcommand{\inte}{\mathrm{int}}
\newcommand{\A}{\mathbb{A}}

\hypersetup{pdflinkmargin=2pt}

\newtheorem{theorem}{Theorem}[section]
\newtheorem{proposition}[theorem]{Proposition}
\newtheorem{corollary}[theorem]{Corollary}
\newtheorem{lemma}[theorem]{Lemma}

\theoremstyle{definition}

\newtheorem{definition}[theorem]{Definition}
\newtheorem{remark}[theorem]{Remark}
\newtheorem{example}[theorem]{Example}

\newtheorem{problem}[theorem]{Problem}

\begin{document}

\maketitle

\begin{abstract}
   We study Voronoi diagrams of manifolds and varieties with respect to polyhedral norms.
   We provide upper and lower bounds on the dimensions of Voronoi cells. 
   For algebraic varieties, we count their full-dimensional Voronoi cells.
   As an application, we consider the polyhedral Wasserstein distance between discrete probability distributions.

\end{abstract}

%\tableofcontents

	\section{Introduction}
	
	Given a metric space $(M,f)$ and a subset $X\subseteq M$, its \emph{Voronoi diagram} is the collection of \emph{Voronoi cells}  indexed by points of $X$.  The cell corresponding to $x\in X$ is given by all points of $M$ which are closer to $x$ than to any other point in $X$. 
	The first appearance of this idea in the sciences can be traced back to the work of Descartes and his interpretation of the solar system as a union of convex regions corresponding to fixed stars. 
	At the beginning of the 20th century, the mathematician Georgy Fedoseevich Voronoi gave a formal definition of the diagrams which now carry his name. 
	Since then, Voronoi diagrams found a long list of applications and specializations in different areas: from electrical engineering, network and data analysis to medicine.
	
	The most popular version of Voronoi diagrams consists of the following choice of ingredients: $M=\mathbb{R}^n$, $f$ is the Euclidean distance on $M$, and $X$ is a finite set of points or a lattice. %In this setting there has been a lot of attention on the algorithmic and computational aspects. 
	However, our focus is on a different setting:
	
	We let the ambient space $M$ to be an affine space with a a \emph{polyhedral distance}, i.e., a distance for which the unit balls are convex polytopes. 
	We shall not only consider countable subsets $X$ of points, but higher dimensional manifolds and varieties. 
	This choice produces a Voronoi diagram with an infinite number of cells as these are indexed by points in $X$.
	
	Even though most of our technical results hold true in this general setting, we will focus on the following motivating application. 
	Let us consider the affine hyperplane $\mathbf{1}_n:=\{(t_1,\dots,t_{n+1}): \sum_i t_i = 1\}\subseteq\mathbb{R}^{n+1}$
	and the \emph{probability simplex} $\Delta_{n}:=\mathbf{1}_n\cap \mathbb{R}^{n+1}_{\geq 0}$.
 The latter is the space of discrete probability distributions on $n+1$ states. 
 We consider a \emph{Wasserstein distance} $W_d$ on $\mathbf{1}_n$ (or its restriction to $\Delta_{n}$). 
 The distance $W_d$ is determined by a metric $d$ on the finite set of states $\{1,\dots,n+1\}$. 
 We can interpret $d$ as describing the cost of transporting a unit of mass from one state to another. Then, the Wasserstein distance $W_d(\mu,\nu)$ is the minimal cost of transforming the distribution $\mu$ into the distribution $\nu$ by moving mass between the different states. 
 
 In more mathematical terms, $W_d(\mu,\nu)$ can be computed by optimizing a linear cost function (determined by $d$ alone) over the convex \emph{transportation polytope} (determined by $\mu$ and $\nu$). 
 The notion of Wasserstein distance is at the core of the field of \emph{optimal transport} \cite{villani08}, and gained recently popularity in \emph{machine learning}, where it has been successfully employed as a loss function for generative adversarial networks (WGANs) \cite{arjovsky2017wasserstein,10.5555/2969442.2969469}. 
 For us, the key observation in the setting of discrete probability distributions is that Wasserstein unit balls are convex polytopes, known in the literature as \emph{Kantorovich-Rubinstein} or \emph{fundamental} polytopes \cite{Ver, GP}.
	
Many statistical applications deal with the problem of 
determining which distribution in a given \emph{statistical model} $X \subsetneq \Delta_n$ best explains some observed data.
That problem can typically be phrased as finding a  distribution in $X$ that is closest to an observed distribution.
 In \cite{OJMSV,CJMSV}, the authors study the problem of describing and computing the Wasserstein distance to a model $X$ that is an \emph{algebraic variety}, i.e., the zero locus of a finite set of polynomials. This assumption is satisfied by a wide variety of models in statistics, such as models of independence and discrete exponential families. Our study can be seen as a reformulation of the question above:
	\begin{center}
		\emph{Given a point $x$ in a statistical model $X$, what is the set of empirical distributions that are better explained, in the sense of Wasserstein distance, by $x$ than by any other point in the model?}
	\end{center}
	
	In \cite{CIFUENTES2020}, the authors study a similar question, when the distance is the Euclidean one. They prove that the Euclidean Voronoi cell at a smooth point $x\in X$ is convex, and contained in the normal space of $X$ at $x$. 
	Moreover, this cell is full-dimensional in the normal space. If the point $x$ is singular, then the corresponding Voronoi cell can be full-dimensional in the ambient space of $X$. 
	In contrast, when using Wasserstein distances, we observe that full-dimensional cells can occur even at smooth points. 
	This is due to the geoemetric difference between the corresponding unit balls: while the boundary of the Euclidean unit ball is smooth, a convex polytope has only a piecewise smooth boundary. 
	Other related works study the special case of plane curves \cite{brandt2019voronoi}, 
	Voronoi cells arising from maximum likelihood estimation \cite{alexandr2021logarithmic,alexandr2022logarithmic}, and Voronoi diagrams of points in the context of tropical geometry \cite{CJS}. The distance considered in \cite{CJS} corresponds to a Wasserstein distance (for a specific choice of $d$) via  Kantorovich-Rubinstein duality.
	
	After the necessary preliminaries and definitions in Section \ref{sec: 2}, we study the dimensions of Voronoi cells depending on the geometry of the set $X$ in Section \ref{sec: 3}. 
	We consider a specific example \Cref{sec: 4}: We compute the number of full-dimensional Voronoi cells of the \emph{Hardy-Weinberg} curve in $\Delta_2$ with respect to \textit{every} Wasserstein distance. Finally, in \Cref{sec: 5}, we provide an upper bound for the number of full-dimensional Voronoi cells when the set $X$ is an algebraic variety. Although for a general polyhedral distance this bound depends also on the combinatorics of the unit ball, in the case of Wasserstein distances it can be relaxed and made dependent only on the dimension of the ambient space and on the geometry of the variety.

%\section{Introduction and Outline}
%In the second chapter of this thesis we introduce the concept of the Wasserstein distance. We give some historical background and define the Wassertein distance balls. We also provide some background on Voronoi diagrams and define Voronoi cells.

%The third chapter is dedicated to understanding Voronoi cells w.r.t Wasserstein distances. We introduce face-cones which allow us to break down Voronoi cells into more manageable components. We prove a series of results which produce upper and lower bounds for dimensions of Voronoi cells depending on the orientation of faces of the Wasserstein distance ball.

%We spend the fourth chapter more thoroughly exploring the Voronoi diagrams of the statistically significant Hardy-Weinberg curve in the simplex w.r.t different Wasserstein metrics.

%The final chapter is where we find an upper bound for the number of full-dimensional Voronoi cells for smooth irreducible varieties with hypersurfaces as dual varieties. We use dual varieties and polar degrees for the upper bound w.r.t. a generic  Wasserstein distance. We also reference adjecent research and properties of generic distance balls and provide an example when our bound is sharp.
%\newpage

\section{Polyhedral norms and Wasserstein distances}
\label{sec: 2}
Let $\A^n$ denote a real affine $n$-dimensional  space. We will consider affine $n$-dimensional spaces embedded in $\mathbb{R}^{n+1}$. For a metric $D$ on $\A^n$, we consider the \emph{closed $D$-balls} $B_{D,r}(c) := \lbrace x\in \A^n \,:\, D(x,c) \leq r \rbrace$ with center $c \in \A^n$ and radius $r \geq 0$.
If the center and radius of a closed ball are not relevant, we simply write $B_D$.
\begin{definition}
    The metric $D$ is a \emph{polyhedral distance} if it is translation invariant and the $D$-balls $B_{D,r}(c)$ of positive radius are convex $n$-dimensional polytopes. 
\end{definition}
We note that the polytopes $B_{D,r}(c)$ associated with a polyhedral distance are centrally symmetric.
Hence, each face $F$ of the polytope $B_{D,r}(c)$ has an \emph{opposite face}, which we denote by $-F$.
Basic examples of polyhedral distances are those induced by the $L_1$- and the $L_{\infty}$-norm on $\mathbb{R}^n$.

%According to \textit{Encyclopedia of Mathematics}  \cite{wasser} the Wasserstein distance first appears in writing as "Vasershtein distance" in 1970 in a discussion on probability measures. The name then referred to a paper from 1969 by Vasershtein on Markov processes. However the earliest appearence of the distance was in 1940 in a paper by Kantorovich on transportation problems.

Another large family of polyhedral distances is given by Wasserstein distances, defined on $\mathbf{1}_n$.
Let $d$ be a metric on the finite set $[n+1]:=\{1,\dots,n+1\}$.
We can naturally identify $d$ with a symmetric matrix $d=(d_{ij})\in\mathbb{R}_{\geq 0}^{(n+1)\times(n+1)}$ that satisfies the triangular inequalities $d_{ij}\leq d_{ik}+d_{kj}$ for every $1\leq i,j,k\leq n+1$. Given $\mu, \nu\in \mathbf{1}_n$, the \emph{transportation polytope} from $\mu$ to $\nu$ is

\[
    \Pi(\mu,\nu) :=\left\lbrace y\in \mathbb{R}_{\geq 0}^{(n+1)\times(n+1)} \; : \; \forall\, 1\leq i,j\leq n+1, \; \sum_{k=1}^{n+1} y_{ik}=\mu_i \text{ and } \sum_{k=1}^{n+1} y_{kj}=\nu_j\right\rbrace.
\]
\begin{definition}
    The \emph{Wasserstein distance} $W_d(\mu,\nu)$ of $\mu$ and $\nu$ associated with $d$ is the solution to the following linear optimization problem: 
\begin{align*}
    &\text{minimize} \sum_{i,j=1}^{n+1} d_{ij}x_{ij}\\
    &\text{subject to } x=(x_{ij})\in\Pi(\mu,\nu).
\end{align*}
\end{definition}
If we interpret $\mu$ and $\nu$ as mass distributions, $x$ can be thought of as a transportation plan, in which the entry $x_{ij}$ indicates the amount of mass that is transported from $\mu_i$ to $\nu_j$ in order to transform one distribution into the other. The matrix $d$ determines the cost of each transportation step.
%The unit balls associated with the metric $W_d(\cdot,\cdot)$ on $\mathbf{1}_n$ are the following centrally symmetric convex polytopes.

The \emph{Wasserstein balls} (i.e.,  the closed $W_d$-balls associated with a Wasserstein distance $W_d$ on $\boldsymbol{1}_n$) can be expressed as a convex hull of $2\binom{n+1}{2}$ points:
\begin{align*}
%    \label{eq:wassersteinball}
    B_{W_d,r}(c)=\text{conv}\left\lbrace c+r\frac{e_i-e_j}{d_{ij}}~:~ ~i,j\in [n+1], i\neq j\right\rbrace,
\end{align*}
where $e_i$ denotes the $i$-th vector of the standard basis of $\mathbb{R}^{n+1}$.
Note that those $2\binom{n+1}{2}$ points do not need to be in convex position, as the following example demonstrates.

\begin{example}
\label{example1}
\Cref{fig:hexagonball} shows the Wasserstein balls of radius $r=\frac{1}{3}$ and center $c=(\frac{1}{3},\frac{1}{3},\frac{1}{3})$ associated with the distances
\[
d_1 = \begin{pmatrix}
0 & 1 & 1\\
1 & 0 & 1\\
1 & 1 & 0
\end{pmatrix} \quad \text{and} \quad 
d_2 = \begin{pmatrix}
0 & 1 & 2\\
1 & 0 & 1\\
2 & 1 & 0
\end{pmatrix}.
\]

\begin{figure}[H]
    \centering
   \includegraphics[scale=0.5]{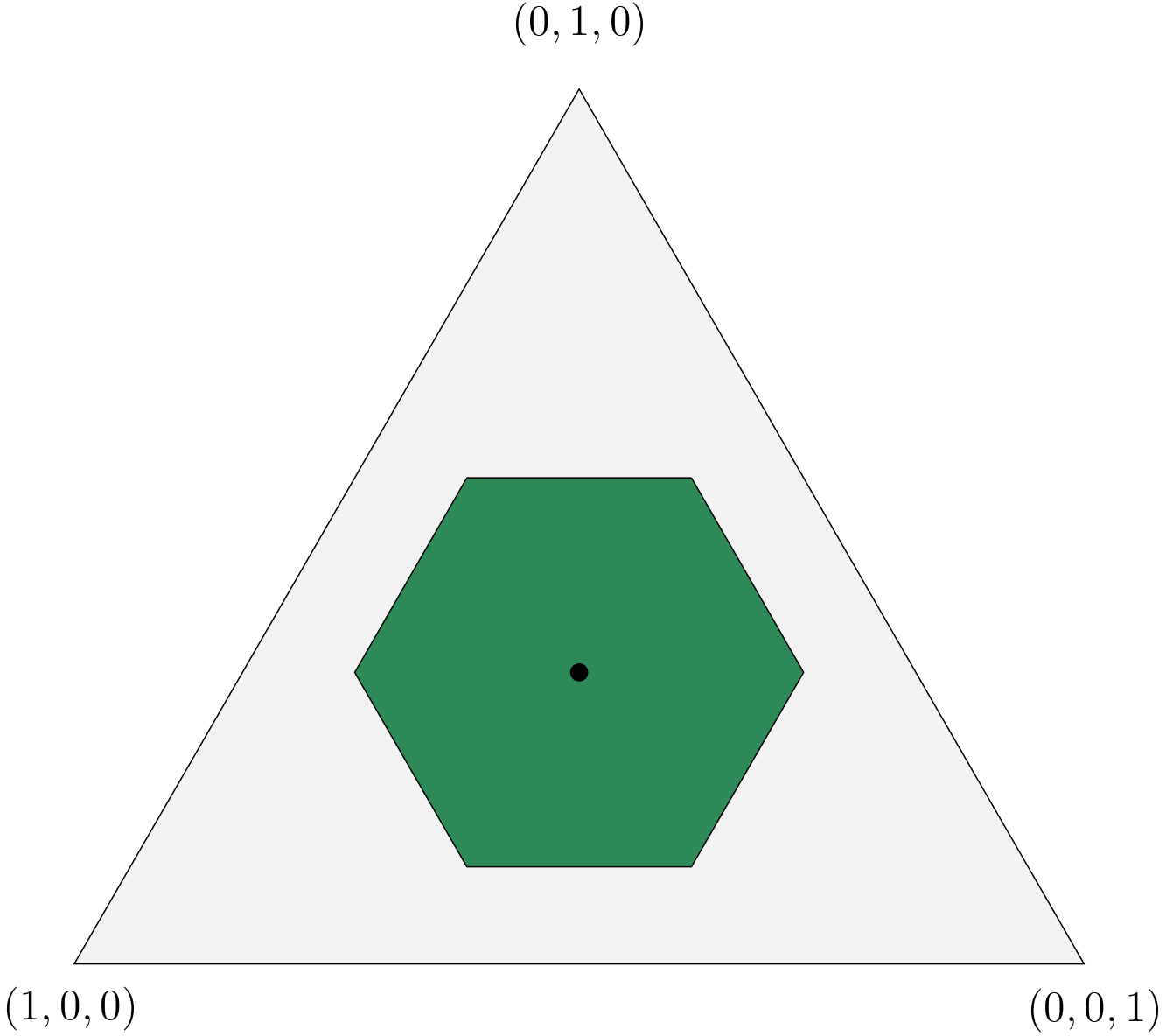} 
   \quad\quad
   \includegraphics[scale=0.5]{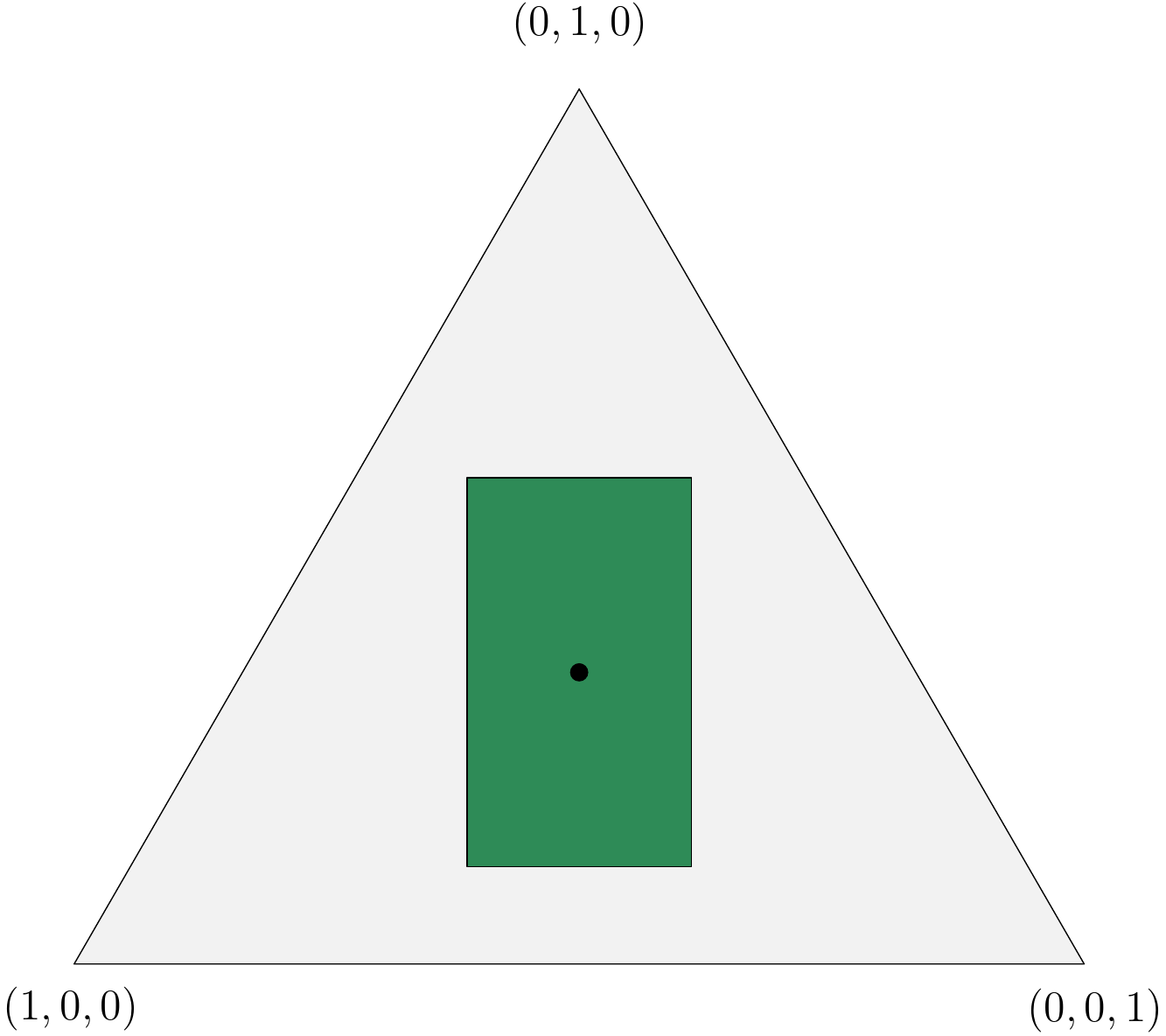}
    \caption{Wasserstein balls $B_{W_{d_1},\frac{1}{3}}(\frac{1}{3},\frac{1}{3},\frac{1}{3})$ and $B_{W_{d_2},\frac{1}{3}}(\frac{1}{3},\frac{1}{3},\frac{1}{3})$.}
    \label{fig:hexagonball}
\end{figure}

These are all cases for the possible number of sides of Wasserstein balls in the plane $\mathbf{1}_2$. Indeed, planar Wasserstein balls are the convex hulls of six points. As centrally symmetric polytopes have an even number of vertices, that number must be either four or six.
\end{example}

%\begin{proposition}
%\label{ballshape}
%Wasserstein distance balls in $\mathbf{1}_2$ are polygons with six or four sides or, when the radius is 0, a single point. 
%\end{proposition}
%\begin{proof}
%Since the Wasserstein distance balls $B_{D,r}(c)$ in $\mathbf{1}_2$ are the convex hulls of 6 points in a plane they can only be 6,5,4 or 3 sided polygons, a line segment or a point.
%Because the 6 points are endpoints of evenly spaced vectors around c, with length $\frac{r}{d_{i,j}}$ the line segment can be ruled out and the point only achieved when r=0.
%The symmetry of the vectors $c+\frac{u_i-u_j}{d_{i,j}},c+\frac{u_j-u_i}{d_{j,i}}$ means we have an even number of edges/vertices when $r\neq 0$.
%\end{proof}

%\begin{proof}
%If a point x is closer to a point y on the curve than any other point on the curve then some ball around x contains y and no other point on the curve. If in addition the tangent plane at x is parallel to some face on the ball(orientation?) then the ball of half radius centered halfway between x and y will be contained in the larger ball and contain x, any translation of the ball along the tangent plane of at most half the original diameter will still be contained in the original ball and contain x.
%\end{proof}

\section{Dimensions of Voronoi cells}
\label{sec: 3}
Our main object of study is the Voronoi decomposition of a set $X$ under a polyhedral distance.
We are particularly interested in the case of an algebraic variety $X$ under a Wasserstein distance. 
This section is devoted to finding  lower and upper bounds on the dimensions of the Voronoi cells.

\begin{definition}
Let $(M,f)$ be a metric space and $X\subseteq M$ be a set.
The (open) \emph{Voronoi cell} $V_{f,X}(x)$ of a point $x\in X$ is the set $$V_{f,X}(x)=\{y\in M ~:~ \forall x'\in X \setminus \lbrace x \rbrace, f(x,y)< f(x',y)\}.$$

\end{definition}

From now on we fix a polyhedral distance $f=D$ on real affine $n$-space $M = \A^n$.
An important tool in our analysis are the face cones of the closed $D$-balls $B_{D,r}(c)$.
The \emph{face cone} of a face $F$ is the cone over the opposite face $-F$.
For the formal definition, we denote by $\inte(\cdot)$ the relative interior.

%We now recall a few definitions from discrete geometry. A \emph{face} $F$ of a polytope $P$ is a subset of $P$ which is the minimizer of any linear functional. The \emph{dimension} $\dim F$ of a face $F$ is the dimension of its affine hull $\text{span}(F)$.

\begin{definition}
\label{conedef}
Let $F$ be a non-empty face of the polytope $B_{D,r}(x)$ for some  $x\in \mathbf{1}_n$ and $r \geq 0$. 
The \emph{face cone} of $F$ at $x$ is the set  $$C_{F}(x):=\{x+\delta(s-x) \; : \; s\in \inte(-F), \delta>0\}.$$  
We set $C_{\emptyset}(x)=\{x\}$. We also define the \emph{truncated face cone} for some $\varepsilon > 0$ as follows: $$C_{F,\varepsilon}(x)=\{x+\delta(s-x) \; : \; s\in \inte(-F), \varepsilon>\delta>0\}.$$
\end{definition}
\begin{figure}[h]
    \centering
    \includegraphics[scale=0.5]{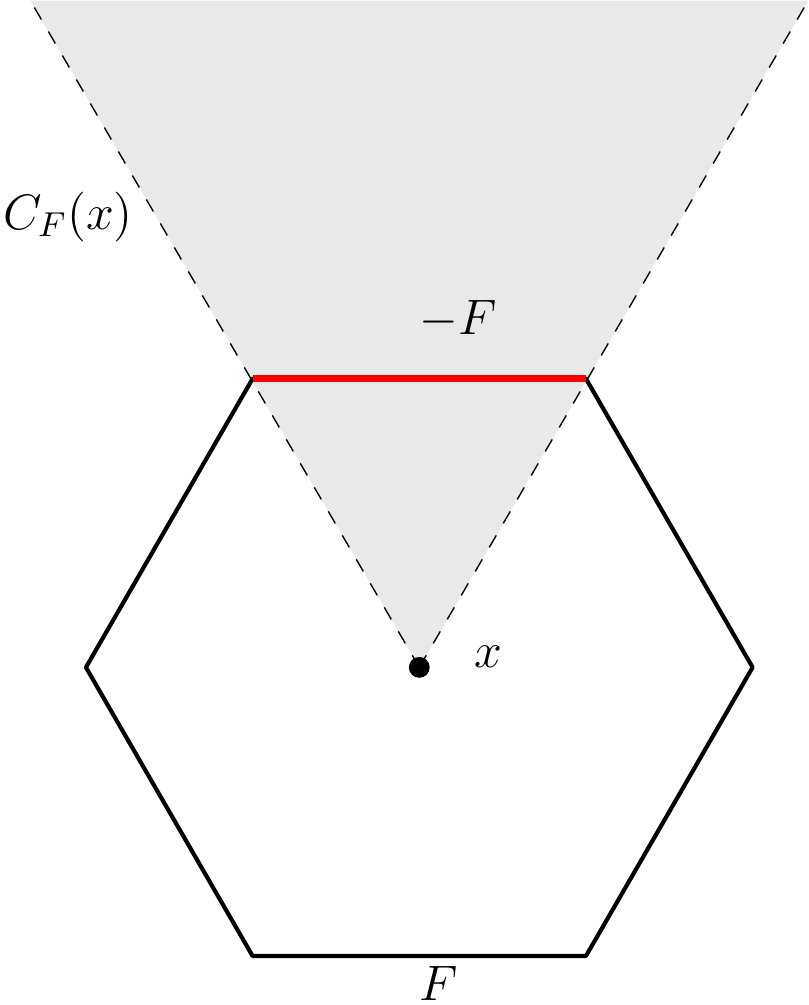}
    \caption{Illustration of the face cone as in Definition \ref{conedef}.}
    \label{fig:cone}
\end{figure}

The face cone of a face $F$ at a point $x$ is the set of all points such that the $D$-ball of some (positive) radius  around these points intersects $x$ at $F$; see Figure \ref{fig:cone}.
We note that the definition of the face cone $C_F(x)$ does not depend on the radius $r$ of the ball $B_{D,r}(x)$ while the truncated face cone $C_{F,\varepsilon}(x)$ does. 
Moreover, we observe that $C_{F,\varepsilon}(x)\subseteq B_{D,\varepsilon}(x).$
%Having the ball $B_{D,r}(x)$ with which $C_f(x)$ is defined be centered at the same point $x$ as the face-cone is not strictly necessary as the vector $s-x$ could be replaced with a vector from the center of any ball to a point on that balls face. 
Most notably, the face cones of the distinct faces of a $D$-ball centered at $x$ partition the Voronoi cell $V_{D,X}(x)$ of $x$, i.e.,
\begin{align}\label{eq: decomposition}
%     V_{W_d,X}(x) = \dot\bigcup\limits_c \left( V_{W_d,X}(x) \cap C_F(x) \right).
 V_{D,X}(x) = \underset{F \subsetneq B_{D,r}(x) \text{ face} }{\dot\bigcup} \left( V_{D,X}(x) \cap C_F(x) \right).
\end{align}
We make use of this in a first dimension estimate of Voronoi cells.
%The following fact shows the usefulness of knowing in which face-cone a point of the Voronoi cell lies.

\begin{proposition}
\label{nocon}
Let $X\subseteq \A^n$ be a subset and let $x\in X$ be a point. 
Assume that $y\in V_{D,X}(x)$ with $y\neq x$, and let $\varepsilon>0$ be the radius such that  $\partial B_{D,\varepsilon}(y)\cap  X = \{x\}$. Let $F$ be the unique face of $B_{D,\varepsilon}(y)$ whose relative interior contains $x$. Then
$
    \dim V_{D,X}(x) \geq \dim F + 1.
$
\end{proposition}
\begin{proof}
We claim that $C_{-F,\varepsilon}(y)\cap C_{F}(x)\subseteq V_{D,X}(x)$. Let $z$ be a point in the intersection. Since both $x\in C_{-F}(y)$  and $z\in C_{-F,\varepsilon}(y)\cap C_{F}(x)$, the polyhedral distances 
$D(x,y)$, $D(x,z)$ and $D(y,z)$ are measured evaluating the same linear functional. In particular, it holds that $D(x,y)=D(x,z)+D(y,z)$. %This holds since for a point $z$ in the intersection, there exists a $\delta>0$ such that $x\in B_{d,\delta}(z)\subset B_{d,\varepsilon}(y)$, which together with the hypothesis $B_{d,\varepsilon}(y)\cap  X = \{x\}$, implies that $B_{d,\delta}(z)\cap  X = \{x\}$.
%To prove the claim we let $\delta=W_d(z,x)$ and $p\in B_{d,\delta}(z)$. We then have $W_d(y,p)\leq W_d(y,z)+W_d(z,p)\leq (\varepsilon-\delta)+\delta=\varepsilon$, which follows from the fact that  {\color{red}$W_d(y,z)+W_d(x,z)=W_d(y,x)=\varepsilon.$ THIS EQUATION NEEDS TO BE BETTER JUSTIFIED.}\\
Assume now that there exists a point $x'\in X$ with $D(x',z)\leq D(x,z)$. Then by the triangular inequality we obtain that 
\begin{align*}
    D(x',y) & \leq D(x',z) + D(y,z)\\
    &\leq D(x,z) + D(x,y)-D(x,z)=D(x,y),
\end{align*}
which contradicts the fact that $y\in V_{D,X}(x)$.

Since  the relative interior of the line segment from $x$ to $y$ lies in both $C_{-F,\varepsilon}(y)$ and $C_{F}(x)$, their intersection is  nonempty.
Moreover, inside the $(\dim F+1)$-dimensional affine space spanned by $F$ and $y$, both $C_{-F,\varepsilon}(y)$ and $C_{F}(x)$ are full-dimensional and open, and so is their intersection.
Hence, the dimension of the Voronoi cell  $V_{D,X}(x)$ is at least $\dim F+1$.
\end{proof}

%{\color{red}\Cref{nocon} shows that finding points in a Voronoi cell belonging to a certain face cone gives information about the dimension of the Voronoi cell. 

The lower bound in \Cref{nocon} depends on finding a point in the Voronoi cell $V_{D,X}(x)$.
We now give another lower bound that only depends on the faces of the $D$-balls. 

We do this by splitting an open neighbourhood around $x$ with a hyperplane that leaves $X$ entirely on one side. 
For this, we consider a Euclidean ball in $\mathbb{R}^{n+1}$  centered at a point $x \in \A^n$ of radius $r$ and denote with $\beta_r(x)$ its restriction to $\A^n$. 

For a proper face $F$ of a $D$-ball,
we define a family of parallel hyperplanes in $\A^n$ as follows:
Let the $D$-ball be centered at $c \in \A^n$.
We denote by $v_F$ the  outer orthonormal vector to the face $F-c=\{ p-c ~:~ p\in F\}$ in the linear space $\text{span}((F-c)\cup \{\mathbf{0}\})$. Here the inner product that defines $v_F$ is the one induced by the standard inner product on $\mathbb{R}^{n+1}$, in which $\A^n$ is embedded. 
We write $H_F$ for the hyperplane in the linear space $\A^n-c$ that is the orthogonal complement of $v_F$.
Finally, for any point $x\in\A^n$,  we define $H_{F,x} := H_F+x$ for the hyperplane in $\A^n$ that passes through $x$. 
If $\dim F = k$, then we call $H_{F,x}$ a \emph{$k$-face hyperplane} through $x$. Note that the definition of $H_{F,x}$ does not depend on the center $c$, and that $H_{F,x} = H_{-F,x}$

%Under certain conditions we can think of the hyperplane $H_{f,x}$ passing through a point $x$ defined as the orthogonal complement of $v_f$ to separate a distance ball and the manifold $x$ belongs to; see Figure \ref{fig:3.4}. 
\begin{proposition}
\label{maximum}
Let $X\subseteq \A^n$ be a subset and $x\in X$ be a point.
If there exists a radius $r>0$ and a $k$-face hyperplane through $x$ that bisects the Euclidean ball $\beta_{r}(x)$ in such a way that the closure of one of the resulting half-balls intersects $X$ only at $x$,
then 
$
    \dim V_{D,X}(x) \geq k + 1.
$

\end{proposition}

\begin{figure}[H]
    \centering
    \includegraphics[scale=0.75]{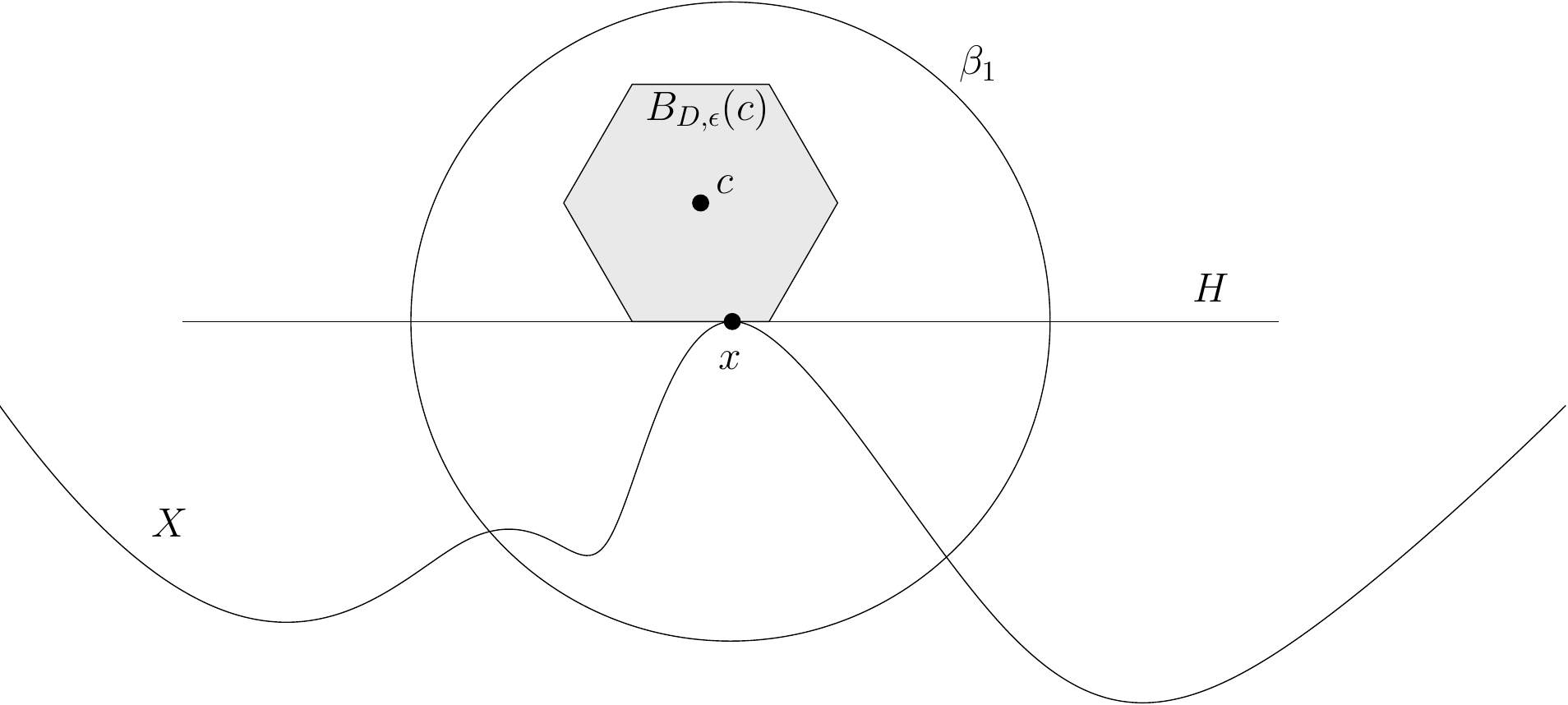}
    \caption{Example illustration of setup in proof of Proposition \ref{maximum} in the plane.}
    \label{fig:3.4}
\end{figure}

\begin{proof}
We denote with $H$ the $k$-face hyperplane that bisects $\beta_r(x)$ into   two closed half-balls, say $\beta_1$ and $\beta_2$.
Without loss of generality, $\beta_1 \cap X = \{ x \}$. 
We choose a point $c \in \beta_1$ and a sufficiently small $\varepsilon > 0$ such that  the following two conditions hold:
1) $B_{D,\varepsilon}(c) \subseteq \beta_1$
and 2) 
$x$ is in the relative interior of a $k$-face $F$ of $B_{D,\varepsilon}(c)$ with $H = H_{F,x}$; see Figure \ref{fig:3.4}.
In particular, this gives us a point $c$ in the Voronoi cell of $x$ such that $\partial B_{D,\varepsilon}(c)\cap  X = \{x\}$.
By Proposition \ref{nocon}, $V_{D,X}(x)$ has dimension at least $k+1.$
%Consider a Wasserstein ball $B_{d,\varepsilon}(c)$ such that $x$ is in the relative interior of $F$. We can choose the point $c \in \mathbf{1}_n$ and $\varepsilon$ small enough that $B_{d,\varepsilon}(c)\subset \beta_{r}(x)$.
%By construction, the hyperplane  $H_{F,x}$ splits $\beta_{r}(x)$ into two open half-balls $\beta'$ and $\beta''$ such that $B_{d,\varepsilon}(c)\cap \beta'=\emptyset$.
%Since $B_{d,\varepsilon}(c)\cap X = \{x\}$ it follows that $c\in V_{W_d,X}(x)$, and since  Otherwise a symmetric construction of $B_{d,\varepsilon}(c)$ such that $x$ is in the relative interior of $ -f$ will give us such a $c.$ 
%By Proposition \ref{nocon}, $V_{W_d,X}(x)$ has dimension at least $\dim(f)+1.$
\end{proof}

%\begin{proof}[bad proof]
%Since $x$ is a local  maximum there is some euclidean ball $E$ centered at $x$ of some radius $r$ such that $x$ is a global maximum for the set $A\cap E $. Let the unit Wasserstein ball have k as maximum euclidean distance from the origin.The face-cone of f at x is by construction the set of points $y_{\varepsilon,f}\in\{x+\varepsilon\overline{s};s\in -f, \varepsilon>0\}$ such that the Wasserstein ball $B_{\varepsilon,s}=\varepsilon B+y_{\varepsilon,s}$ intersects x at a point on f. Furthermore by maximality of x and convexity of B we have that $B_{\varepsilon,s}\cap A \cap E=x$.
%Finally the restriction of  the face-cone to $\{x+\varepsilon\overline{s};s\in -f, \frac{r}{2k}>\varepsilon>0\}$ is a d+1 dimensional object such that all points in the restriction are closer by the Wasserstein distance to x than any other point in $A\cup E^c$. 
%\end{proof}

We now aim to find upper bounds for the dimension of a Voronoi cell.
For this, we will determine which faces do not appear in the decomposition \eqref{eq: decomposition} of a given Voronoi cell. 
We restrict our investigation to sets $X$ that are algebraic varieties, but note that the statements in this section also apply to  manifolds.
Similarily to the construction of the $k$-face hyperplanes $H_{F,x}$, 
we define for every face $F$ of a $D$-ball and for every point $x \in \A^n$ the \emph{affine face space} $A_{F,x}$ to be the affine hull $\mathrm{span}(F)$ translated such that it passes through $x$.

%Assume we have  a smooth manifold $X$ and a point $x\in X$. Furthermore, let $y\in C_f(x)$ and $B_{d,\varepsilon}(y)$ be the distance ball such that $x$ is in the relative interior of the face $f$ of $B_{d,\varepsilon}(y)$.
%Our condition $(C)$ for a plane $P$ is that $x,y\in P$ and that there exists two non-parallel line segments $l_f,l_t\subset P$ such that $x$ is in the relative interior of each line segment, $l_f$ is in the face $f$ of $B_{D,\varepsilon}(y)$ and $l_t$ is in the tangent space of $A$ at $x$. This property might seem arbitrary but having the lemma formulated in this manner will make it easier to prove other statements about Voronoi cells later.

\begin{lemma}
\label{P}
 Let $X \subsetneq \A^n$ be an algebraic variety and $x\in X$ be a smooth point.
 If an affine face space $A_{F,x} \subsetneq \A^n$ through $x$ intersects $X$ at $x$ transversally, then the Voronoi cell $V_{D,X}(x)$  does not contain any point in the face cone $C_{F}(x)$.
\end{lemma}

\begin{proof}
We assume for contradiction that there exists a point $y\in V_{D,X}(x)\cap C_F(x).$
This means that $\partial B_{D,\varepsilon}(y) \cap X = \{ x \}$ for some $\varepsilon > 0$
and that $F$ is the unique face of $B_{D,\varepsilon}(y)$ whose relative interior contains $x$.

We now restrict our problem to the affine span of $A_{F,x}$ and $y$, denoted by $\mathbb{A}$.
Inside $\mathbb{A}$, the affine space $A_{F,x}$ is a hyperplane that intersects $X' := X \cap \mathbb{A}$ at $x$ transversally. 
Note that $x$ is non-singular on $X'$.
Moreover, since $x$ is the closest point to $y$, it is a local minimum for the linear functional defined by the vector $v_F$ orthogonal to $A_{F,x}$. However, this yields a contradiction since smooth critical points of the problem of optimizing a linear functional $\langle v_F,\cdot\rangle$ over an algebraic variety are precisely the points for which the tangent space is parallel to the orthogonal complement of $v_F$.

\end{proof}

By applying Lemma \ref{P} to all faces of large dimension, we obtain an upper bound on the dimension of a Voronoi cell.
\begin{theorem}
\label{transverse}
Let $X \subsetneq \A^n$ be an algebraic variety and $x\in X$ be a smooth point.
If an affine face space $A_{F,x}$ through $x$ intersects $X$ at $x$ non-transversally and no affine face space of larger dimension has this property, 
%Assume that $f$ is a face  of some Wasserstein ball $B_{D}(x)$ intersecting  $A$  at  $x$ non-transversally. If no other face of larger dimension has this property  
then the Voronoi cell $V_{D,X}(x)$ is of dimension at most $\dim F+1$.
\end{theorem}
\begin{proof}
For any face $F'$ with $\dim F' > \dim F$, we have that $V_{D,X}(x) \cap C_F(x) = \emptyset$ by Lemma \ref{P}.
% By Lemma \ref{P} the face-cone component $V_{D,A}(x)\cap C_{f'}^\circ(x)$ associated with $f'$ is empty. 
 Recall that we can partition the Voronoi cell as in \eqref{eq: decomposition}. 
 Hence, we conclude that $\dim V_{D,X}(x) \leq \dim C_F(x) = \dim F +1$.
% All faces of dimension larger than $dim(f)$ intersect $A$ at $x$ transversally, so their corresponding face-cone components are empty. We are left with a Voronoi cell contained in the union of finitely many face-cones of faces of dimension dim(f) or lesser. We conclude that $V_D(x)$ is of dimension at most dim(f)+1.
\end{proof}

%For a smooth manifold of codimension 1, the tangent space at any point is of codimension 1 and non-transversal intersection is reduced to tangency. This gives us the following corollary.

%\begin{corollary}
%\label{tangent}
%Let $x$ be some point on a smooth manifold $A$ of dimension one less than the ambient space. If the face $f$ is the largest dimensional face of some Wasserstein ball such that $x\in f$ and $f$ is contained in the tangent space of  $A$  at $x$
%then the Voronoi cell $V_{D,A}(x)$ is of dimension at most dim(f)+1.
%\end{corollary}

%\begin{proof}
%If the tangent space is not contained in the span of $F$ then $F$ and the tangent vector together span the entire space and the intersection is therefore transverse. By theorem \ref{transverse} the Voronoi cell of the point x is then of dimension at most dim(f), not full. 
%\end{proof}

A first application of Theorem \ref{transverse} is to detect all full-dimensional cells at smooth points in the Voronoi diagram of a variety $X$.
For this, we note that for a facet $F$, the affine face space $A_{F,x}$ equals the facet hyperplane $H_{F,x}$.

\begin{corollary}
\label{facet}
Let $x$ be a smooth point on an algebraic variety $X \subsetneq \A^n$.  If the Voronoi cell $V_{D,X}(x)$ is full-dimensional, then 
%the tangent space of $X$ at $x$ is contained in 
one of the facet hyperplanes through $x$ is tangent to $X$ at $x$.
\end{corollary}

\begin{proof}
By Theorem \ref{transverse}, for a Voronoi cell of dimension $n$, there has to be a face $F$ with $n \leq \dim F + 1$ (i.e., a facet) such that the hyperplane $A_{F,x} = H_{F,x}$ is tangent to $X$ at $x$.
\end{proof}

In general, the tangency of a facet hyperplane is only a necessary -- and not a sufficient -- condition for the full-dimensionality of a Voronoi cell. 
However, we can reverse Corollary \ref{facet} in the following sense: 
Once a hyperplane spanned by a facet $F$ is tangent, the assumptions in Propositions \ref{nocon} and \ref{maximum} become equivalent conditions to that the facet $F$ contributes a full-dimensional part of the Voronoi cell.

%\begin{lemma}
%\label{total}
%Assume we have some point $x$ on a smooth manifold A  and  a facet $F$ of $B_D$ such that $x$ is in the relative interior of $F$ and the tangent space of $A$ at $x$ is contained in the subspace spanned by $F$. Let $B_{D,\varepsilon}(c)$ be a distance ball such that x lies on the interior of F.
%In this situation the following are equivalent:
%\begin{itemize}
    %\item[(1)]  x is the closest point of A to the center c of $B_{D,\varepsilon}(c)$ w.r.t. the Wasserstein distance D.
 %\item[(2)]
 %x is a unique local maximum of A as in Lemma \ref{maximum} .
 % \item[(3)] There is some small neighborhood around x where the subspace spanned by F bisects the neighborhood such that all points of A other than x lie strictly in one of the bisected components. 
  %\item[(4)] The Voronoi cell $V_{D,A}(x)$ is full-dimensional and c is a point in that cell.
%\end{itemize}
%\end{lemma}

\begin{corollary}
\label{total}
Let $x$ be a smooth point on an algebraic variety $X \subsetneq \A^n$ such that the facet hyperplane $H_{F,x}$ is tangent to $X$ at $x$.
Then the following are equivalent:
\begin{itemize}
    \item[$(a)$] $V_{D,X}(x) \cap C_{F}(x) \neq \emptyset$ or $V_{D,X}(x) \cap C_{-F}(x) \neq \emptyset$ (which implies that the Voronoi cell is full-dimensional).
    \item[$(b)$] There exists $y\in V_{D,X}(x)$ and $\varepsilon>0$ such that $\partial B_{D,\varepsilon}(y)\cap X=\{x\}$ and $x$ is contained in the relative interior of a facet $F'$ of $B_{D,\varepsilon}(y)$, with $H_{F',x}=H_{F,x}$ (cf. \Cref{nocon}).
    \item[$(c)$] There exists $r>0$ such that all the points in $X\setminus\{x\}\cap \beta_r(x)$ lie in one open half ball obtained by bisecting the Euclidean ball $\beta_r(x)$  with $H_{F,x}$ (cf. \Cref{maximum}).
\end{itemize}
\end{corollary}
\begin{proof}
In the proof of Proposition \ref{nocon}, we show that the intersection of the Voronoi cell with the face cone $C_{F'}(x)$ is non-empty, which  gives us the implication $(b)$$\Rightarrow$$(a)$. 
Similarly, the proof of Proposition \ref{maximum} yields $(c)$$\Rightarrow$$(a)$.

To show $(a)$$\Rightarrow$$(b)$, consider any point $y\in  V_{D,X}(x) \cap C_{F}(x)$. 
%Since $y\in C_{F}(x)$, then there exists $\varepsilon>0$ s.t. $\partial B_{d,\varepsilon}(y)$ intersects $X$ in $x$ in the relative interior of $F$. Since $y\in V_{W_d,X}(x)$, we have that $\partial B_{d,\varepsilon}(y)\cap X = \{x\}$.\\
Since $y$ is in the face cone of $F$,
the facet $F$ of some $D$-ball centered at $y$ contains $x$ in its relative interior.
Moreover, since $y$ is also in the Voronoi cell, the boundary of that ball intersects $X$ only at $x$.

To see $(b)$$\Rightarrow$$(c)$, we choose a Euclidean ball $\beta_r(x)$ small enough so that its intersection with $H_{F,x}$ is contained in $F$.
Furthermore, we can choose it small enough such that one of its open half balls is contained in $B_{D,\varepsilon}(y)$.
Since $\partial B_{D,\varepsilon}(y)\cap X=\{x\}$ and $y\in V_{D,X}(x)$, that half ball does not intersect $X$. Hence, all points in $X \setminus \lbrace x \rbrace \cap \beta_r(x)$ lie in the other half ball.
\end{proof}

\begin{remark}
For arbitrary (lower-dimensional) faces,
a weaker version of Corollary \ref{total} holds.
Let us consider a smooth point $x$ on a variety $X$ and an affine face space $A_{F,x}$ of maximal dimension such that it intersects $X$ at $x$ non-transversally. 
Recall that  $\dim V_{D,X}(x)\leq \dim F + 1$ by Theorem \ref{transverse}.
In this setting, 
the conditions $(a)$ and $(b)$ in Corollary \ref{total} are still equivalent, after replacing 
``full-dimensional'' in $(a)$ with ``$(\dim F +1)$-dimensional''
as well as ``facet'' and ``$H_{F',x}=H_{F,x}$'' in $(b)$ with ``face'' resp. ``$A_{F',x}=A_{F,x}$''. 
This can be seen with the same proofs as above.
Similarly, $(c)$ still implies $(a)$ and $(b)$.
However, the reverse implication is generally not true.
A simple counterexample is a plane curve $X$ with a smooth point $x$ of inflection and an appropriately chosen $D$-ball with vertex $F$.
\end{remark}

\section{Wasserstein-Voronoi cells for the Hardy-Weinberg curve}
\label{sec: 4}

We now turn our attention to a concrete example. We examine full-dimensional Voronoi cells for the planar \emph{Hardy-Weinberg curve} under varying Wasserstein distances. 
We show that any Wasserstein-Voronoi diagram of that curve has one, two, or three full-dimensional Voronoi cells.
The Hardy-Weinberg curve is the set of distributions we find by recording the total number of ``heads" obtained from tossing a biased coin twice.
It is parametrized by the bias of the coin, i.e., the probability of the outcome of a single experiment being ``head".
Tossing the coin more than twice, we obtain the \emph{Veronese curves} that are parametrized by   
\begin{align}\label{eq: parametrization Veronese}
    \varphi_n~:~&[0,1]\to \Delta_{n}, \nonumber\\ 
    & p\mapsto 
    \left(p^n, np^{n-1}(1-p),\binom{n}{2}p^{n-2}(1-p)^2,\dots,\binom{n}{n-1}p(1-p)^{n-1},\binom{n}{n}(1-p)^n\right),
    %\left(\binom{n}{i}p^i(1-p)^{n-i}\right)_{i=0}^n
\end{align}
where $\Delta_{n}=\{(t_0,\cdots, t_n)\in \mathbb{R}^{n+1}_{\geq 0}~:~t_0+\cdots+t_n=1\}.$

In order to detect all full-dimensional Voronoi cells of the Hardy-Weinberg curve $\varphi_2([0,1])$, we make use of \Cref{facet} which requires us to find all facet hyperplanes (in our case, \emph{edge lines}) of a given Wasserstein ball that are tangent to the curve.
Recall that planar Wasserstein balls have at most six edges (see Example \ref{example1}).

\begin{lemma}\label{lem: HW cases}
    Let $d = \begin{pmatrix}
0 & d_{12} & d_{13}\\
d_{12} & 0 & d_{23}\\
d_{13} & d_{23} & 0
\end{pmatrix}$. Each of the three directional vectors of the edges of the Wasserstein ball of $d$ is tangent at at most one interior point of the Hardy-Weinberg curve (i.e., to $\varphi_2((0,1))$).
Tangency occurs under the following conditions:
\begin{itemize}
    \item[$(a)$] The vector $\dfrac{e_1-e_2}{d_{12}}-\dfrac{e_1-e_3}{d_{13}}$ is tangent if and only if $d_{12}>d_{13}$.
    \item[$(b)$] The vector $\dfrac{e_1-e_3}{d_{13}}-\dfrac{e_2-e_3}{d_{23}}$ is tangent if and only if $d_{23}>d_{13}$.
    \item[$(c)$] The vector $\dfrac{e_1-e_2}{d_{12}}- \dfrac{e_3-e_2}{d_{23}}$ is  tangent for any $d$.
\end{itemize}
\end{lemma}
\begin{proof}
The tangent vector at the point $\varphi_2(p)$ of the curve is parametrized by 
\begin{equation}\label{eq: tangent vectors to HW}
\left(2p,\, 2-4p,\, 2p-2\right).
\end{equation}

We start by considering the edge with directional vector as in $(a)$. The points $\varphi_2(p^*)$ of the curve whose tangent vector is parallel to the edge are the solutions of the system
\[
 (2p,\, 2-4p,\,2p-2)=\lambda(d_{12}-d_{13},\,d_{13},\,-d_{12}),
\]
for some $\lambda \neq 0$, from which we obtain

\[
    p^* = \dfrac{d_{12}-d_{13}}{2d_{12}-d_{13}}.
\]

In particular, $p^*\in (0,1)$ if and only if $d_{12}>d_{13}$. Statement $(b)$ follows by symmetry, as the Hardy-Weinberg curve is invariant under permutation of the first and the third coordinate.

For the edge specified in $(c)$,  the points $\varphi_2(p^*)$ of the curve whose tangent vector is parallel are the solutions of the system
\[
 (2p, 2-4p,2p-2)=\lambda(d_{23},d_{12}-d_{23},-d_{12})
\]
for some $\lambda \neq 0$, from which we obtain
\[
    p^* = \dfrac{d_{23}}{d_{12}+d_{23}}.
\]
Since $d_{ij}> 0$, we see that $p^*\in (0,1)$.
\end{proof}

We now divide the space $(d_{12},d_{13},d_{23})$ of metrics on $[3]= \lbrace 1,2,3 \rbrace$ according to the number of full-dimensional Voronoi cells  of the Hardy-Weinberg curve. 
\begin{theorem}
\label{lemmaloong}
The number of full-dimensional Voronoi cells of a Wasserstein-Voronoi diagram of the Hardy-Weinberg curve in the simplex $\Delta_2$ is 
\[
    \begin{cases}
        1 & \text{if } d_{13}>\max\{d_{12},d_{23}\}\\
        2 & \text{if } d_{12}<d_{13}<d_{23} \text{ or } d_{23}<d_{13}<d_{12}\\
        3 & \text{if } d_{13}<\min\{d_{12},d_{23}\}
    \end{cases}.
\]
\end{theorem}

\begin{proof}
We know from Corollary \ref{facet} that full-dimensional Voronoi cells can only appear if the tangent line of a point on the curve is the affine hull of an edge of the Wasserstein  ball. 
Since the Hardy-Weinberg curve is a smooth conic, it lies in one of the half-planes bisected by the tangent line at any point.
Hence, by Corollary \ref{total}$(c)$, there is a full-dimensional Voronoi cell at a point if and only if its tangent vector is parallel to an edge.
%By \Cref{lem: HW cases} we see that all points on the curve in the simplex have distinct tangent vectors. 
 \Cref{lem: HW cases} provides linear conditions describing this tangency for each of the edge directions.
 %of the edges of the Wasserstein ball to be tangent to the curve in the interior of the simplex. 
 The statement follows combining those conditions.
\end{proof}

We note that all three cases in Theorem \ref{lemmaloong} can occur for hexagonal Wasserstein balls (i.e., not just for degenerate balls with four edges), as the following example demonstrates.

\begin{example}\label{ex: example Wasserstein}
Under the three Wasserstein distances associated with the metrics
\[
d_1 = \begin{pmatrix}
0 & 1 & 1\\
1 & 0 & 1\\
1 & 1 & 0
\end{pmatrix},\quad 
d_2 = \begin{pmatrix}
0 & 2 & 3\\
2 & 0 & 4\\
3 & 4 & 0
\end{pmatrix},\quad \text{ and }
d_3 = \begin{pmatrix}
0 & 2 & 1\\
2 & 0 & 2\\
1 & 2 & 0
\end{pmatrix},
\]
the Hardy-Weinberg curve has one, two, or three full-dimensional Voronoi cells, respectively.
This can be seen immediately from Theorem \ref{lemmaloong}.
Following the calculations in the proof of Lemma \ref{lem: HW cases}, 
we can find the full-dimensional Voronoi cells at the points
$p = \dfrac{1}{2}$,  $p \in \lbrace \dfrac{2}{3}, \dfrac{4}{5} \rbrace$, and $p \in\lbrace  \dfrac{1}{3}, \dfrac{1}{2}, \dfrac{2}{3} \rbrace$, respectively;
see \Cref{fig: examples full dim}.
\end{example}

\begin{figure}[H]
    \centering
    \includegraphics[scale=0.5]{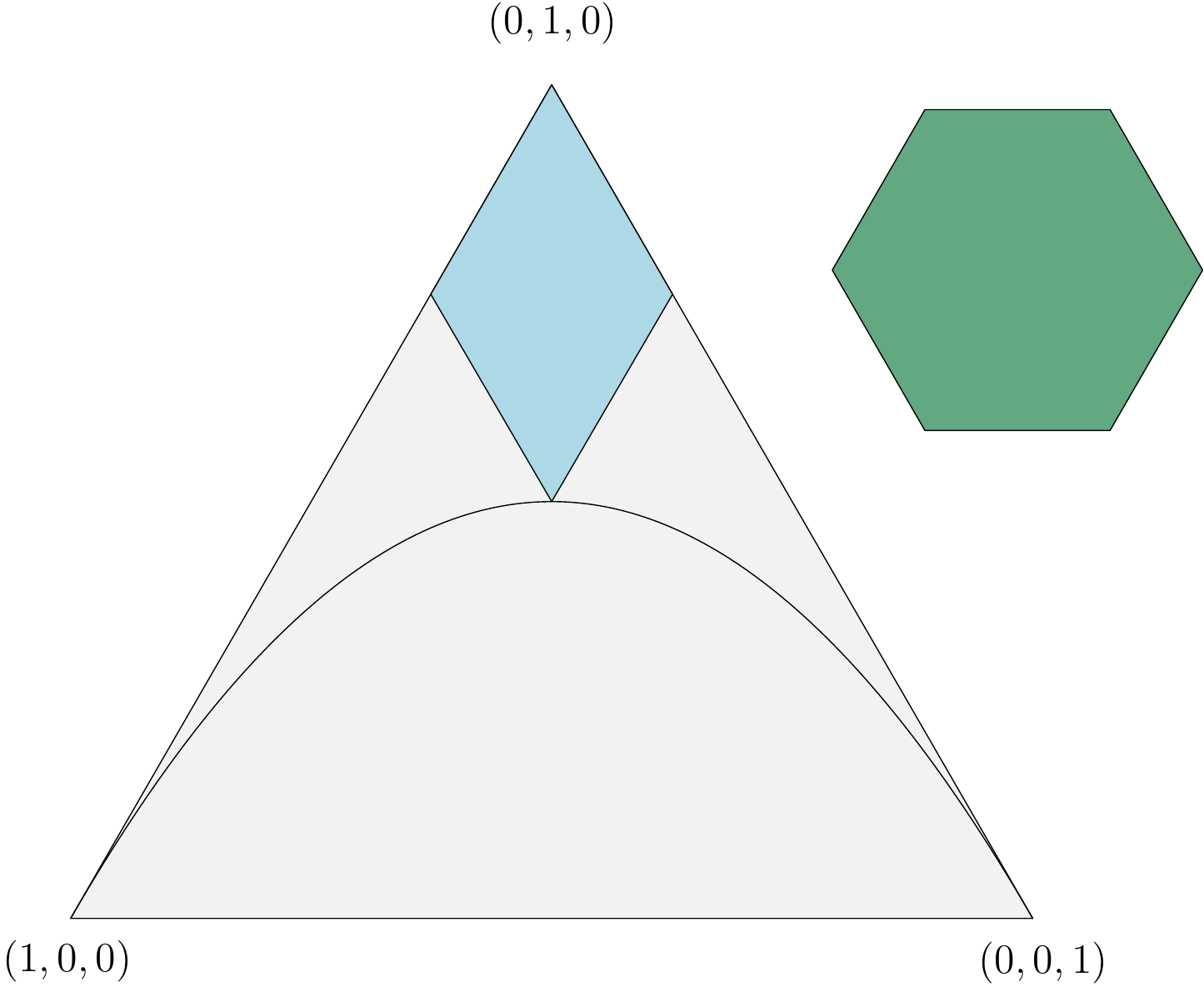}\quad\includegraphics[scale=0.5]{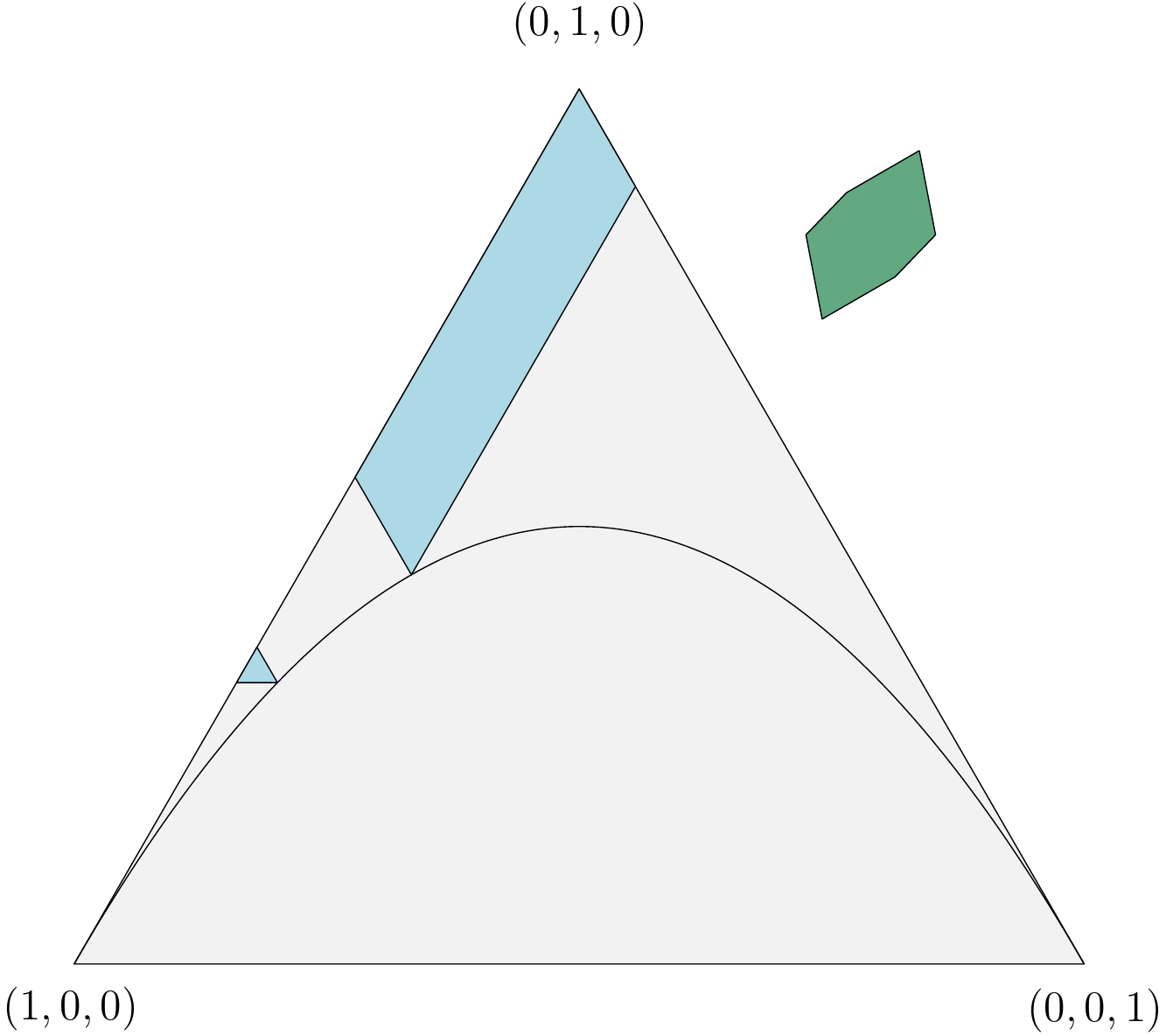}\quad\includegraphics[scale=0.5]{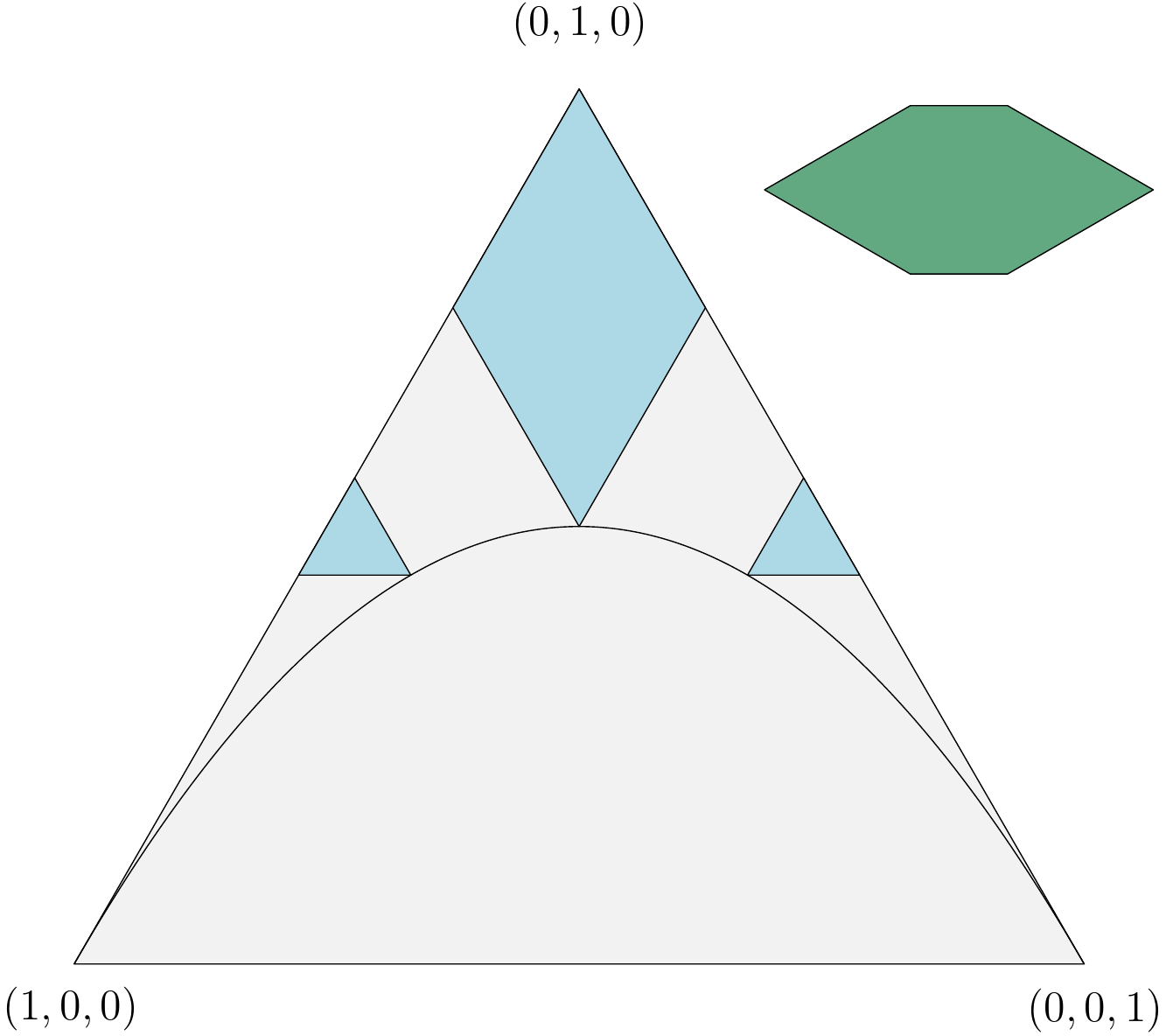}
    \caption{Full-dimensional Voronoi cells of the Hardy-Weinberg curve under the Wasserstein distances $W_{d_1}$ (top left), $W_{d_2}$ (top right), and $W_{d_3}$ (bottom row) from \Cref{ex: example Wasserstein}. 
    The respective Wasserstein balls are depicted in green.}
    \label{fig: examples full dim}
\end{figure}

In the toy example of the Hardy-Weinberg curve, \Cref{lemmaloong} gives a complete understanding of the space of Wasserstein metrics in terms of the number of full-dimensional Voronoi cells. A result of this form would be desirable for 
any family of statistical models. We propose the following problem.
\begin{problem}
    Let $C_n=\varphi_n([0,1])\subseteq \Delta_{n}$ be the Veronese curve parametrized as in \eqref{eq: parametrization Veronese}. Find a decomposition of the space of metrics on $[n+1]$ such that the metrics in every region yield Wasserstein-Voronoi diagrams with the same number of full-dimensional cells. 
\end{problem}

\section{Counting full-dimensional cells}
\label{sec: 5}
In this section, we find an upper bound for the number of full-dimensional Voronoi cells of smooth irreducible varieties under a polyhedral distance $D$ on $\A^n$. 
Our main tool is projective duality.
%We then present a separate but connected approach in the language of polar varieties and polar degrees that has been used in studying optimisation problems for Wasserstein distances in  \cite{CJMSV}.
%For reference of the theory on dual varieties used in this chapter see \cite{polarbook}.
%
%So far we have studied Voronoi diagrams of manifolds  in affine spaces. 
Hence, even though we are interested in real affine varieties $X\subseteq \A^n$, we pass to their complex projective closure $\overline{X}\subseteq\mathbb{P}^{n}$. We denote by $(\mathbb{P}^{n})^*$ the \emph{dual} projective space, i.e., the set of hyperplanes of $\mathbb{P}^{n}$.
%However to apply all the theory in this chapter we need to work with closed irreducible varieties in complex projective space. We will however be able to draw conclusion about our real affine variety $X\subset \mathbf{1}_n$ from what we learn about its complex projective closure $\overline{X}.$
%\section{Dual variety}
%We can find a projective space $P^*$ for a projective space $P$ by associating hyperplanes in $P$ with points in the projective space $P^*$, and a point in $P$ with the hyperplane in $P^*$ of all hyperplanes in through the given point. This in fact gives us a duality between the two spaces where $P^*$ is known as the projective dual space of $P$.
For a subvariety $\overline{X} \subseteq \mathbb{P}^n$, we consider  the set of all hyperplanes that are tangent to $\overline{X}$ at some smooth point.
The \emph{dual variety} $\overline{X}^\vee \subseteq (\mathbb{P}^n)^*$ of $X$ is the Zariski closure of that set inside $(\mathbb{P}^n)^*$.

%The closed irreducible projective variety $\overline{X}\subset P$  can now be said to have a projective dual variety $\overline{X}^\vee\in P^*$, the Zariski closure of the set of all hyperplanes tangent to $\overline{X}$ at some smooth point.

To state the main theorem of this section, we also introduce the notation $F(P)$ for the number of facets of a polytope $P \subseteq \A^n$.

 \begin{theorem}
 \label{polarmax}
Let $X\subseteq \A^n$ be a smooth irreducible variety such that the dual variety $\overline{X}^\vee$ of its complex projective closure $\overline{X}$ is a hypersurface in $(\mathbb{P}^n)^*$. 
The number of full-dimensional Voronoi cells of $X$ under a polyhedral distance $D$ is at most
\begin{align}
\label{eq:upperBoundCounting}
    \frac{F(B_D) \, \deg(\overline{X}^\vee)}{2}.
\end{align}
 \end{theorem}

Before we give a proof for Theorem \ref{polarmax}, we investigate its assumption that the dual variety is a hypersurface.
We show now that there are no full-dimensional Voronoi cells without this assumption.
Thus, Theorem \ref{polarmax} captures all smooth varieties will full-dimensional cells in its Voronoi diagram.

\begin{theorem}
\label{thm:noFullCell}
Let $X\subseteq \A^n$ be an irreducible variety and $D$ be a polyhedral distance on $\A^n$.
If the dual variety $\overline{X}^\vee$ of the complex projective closure $\overline{X}$ is not a hypersurface in $(\mathbb{P}^n)^*$, 
no smooth point $x \in X$ has 
a full-dimensional Voronoi cell with respect to $D$.
\end{theorem}

\begin{proof}
Since the dual variety has codimension larger than one (say, codimension $c+1$), the complex projective variety $\overline{X}$ is \emph{ruled} (by projective spaces of dimension $c$); in other words, $\overline{X}$ is the union of ($c$-dimensional) projective spaces  \cite[Chapter 1, Corollary 1.2]{gkz}.
Moreover, the real part of $\overline{X}$, including $X$, is ruled by real spaces of positive dimension.
Indeed, for a (generic) real point $p \in \overline{X}$, we consider a real hyperplane tangent at $p$.
That hyperplane corresponds to a real point $q \in \overline{X}^\vee$.
Since the dual variety is not a hypersurface, there is a positive-dimensional family of real hyperplanes  tangent at $q$.
These hyperplanes correspond to real points on $\overline{X}$ where they form a positive-dimensional projective space passing through $p$.

Now let us assume for contradiction that there is a smooth point $x \in X$ with a full-dimensional Voronoi cell.
For any point $y$ in the Voronoi cell, there is a $D$-ball $B_{D,\varepsilon}(y)$ that intersects $X$ exactly at $x$.
By Corollaries \ref{facet} and \ref{total},
$y$ can be chosen such that $x$ is contained in the relative interior of some facet $F$ of the ball $B_{D,\varepsilon}(y)$ and such that $F$ is tangent to $X$ at $x$.
Since $X$ is ruled by real affine  spaces of positive dimension, the tangent space of $X$ at $x$ contains such an affine  space passing through $x$.
In particular, the relative interior of the tangent facet $F$ contains infinitely many points on $X$, which contradicts that $B_{D,\varepsilon}(y) \cap X = \lbrace x \rbrace$.
%
%Let $p\in X$ be a point having a full-dimensional Voronoi cell $V_{W_d,X}(p)$. It is necessary for $p$ to be an interior point of a facet $F$ of some Wasserstein ball $B_D(y)$ such that $B_D(y)\cap X=\{p\}.$ We examine the hyperplane F is contained in which corresponds to a real point $p'$ in the dual space. Since $\overline{X}^\vee$ is not a hypersurface neither is its real part. Because of the low dimension of this variety we can find a family of hyperplanes tangent to $p'$ which correspond to a line on $X$ through $p$. This line shows that $F$ contains  points of $X$ other than $p$ so $B_D(y)\cap X\neq\{p\}.$
\end{proof}

The proof of Theorem \ref{polarmax} for a \emph{generic} polyhedral distance (i.e., such that the facet hyperplanes of $B_D$ are sufficiently generic given the variety $X$) is relatively straightforward (as well will see below). 
To address arbitrary polyhedral distances, we need the following basic lemma from projective geometry.

\begin{lemma}
    \label{lem:degenerateCurve}
    Let $n \geq 2$, let $L \subseteq \mathbb{P}^n$ be a projective subspace of codimension two, and let $C \subseteq \mathbb{P}^n$ be an irreducible curve.
    If all tangent lines of $C$ intersect $L$, the curve $C$ is contained in one of the hyperplanes containing $L$.
\end{lemma}

\begin{proof}
    We consider the projection $\pi: \mathbb{P}^n \dashrightarrow \mathbb{P}^2$ from a generic subspace $P$ of $L$ with codimension one.
    The Zariski closure $C_2$ of $\pi(C)$ is either a point or a plane curve whose tangent lines all pass through the point $\pi(L)$.
    In either case, the dual variety $C_2^\vee \subseteq (\mathbb{P}^2)^*$ is contained in a line.
    Hence, $C_2$ itself must be either a point or a line that passes through $\pi(L)$.
    In particular, the preimage $\pi^{-1}(C_2)$ is a projective subspace of $\mathbb{P}^n$ that is contained in a hyperplane passing through $L$.
    Now the assertion follows from the inclusion $C \subseteq \pi^{-1}(C_2)$. 
\end{proof}

Now we are ready to prove the main theorem of this section.

 \begin{proof}[Proof of Theorem \ref{polarmax}]
If $n=1$, the only variety $X$ satisfying the assumptions is a single point. 
Moreover, the only one-dimensional polytopes are line segments, so $F(B_D)=2$. 
This proves the assertion in dimension one, and we assume $n \geq 2$ in the following.

 We write $k:= \dfrac{F(B_D)}{2}$ for the number of pairs of opposite facets of the $D$-balls $B_D$, and enumerate the facet pairs $\left( (F_1,-F_1), (F_2,-F_2), \ldots, (F_k,-F_k) \right)$. 
For each pair $(F_i,-F_i)$ of opposite facets, we consider its one-dimensional family of parallel hyperplanes.
In the projective space $\mathbb{P}^n$, that is a family of hyperplanes that contain a projective subspace $L_i \subseteq \mathbb{P}^n$ of codimension two that lies at infinity.
In the dual projective space $(\mathbb{P}^n)^*$, the family corresponds to the line $L_i^\vee$.
We slightly abuse notation and write $H \in L_i^\vee$ for a hyperplane $H \subseteq \mathbb{P}^n$ containing $L_i$.

 By Corollary \ref{facet}, if the Voronoi cell of $x\in X$ is full-dimensional, then a hyperplane 
$H \in L_i^\vee$ (for some $1 \leq i \leq k$) contains the tangent space $\mathbb{T}_{x}\overline{X} \subseteq \mathbb{P}^n$ of $\overline{X}$ at $x$.
We now investigate all such occurrences of tangency:
\begin{align*}
    \mathcal{N}_i := \left\lbrace (x,H) \in \overline{X} \times (\mathbb{P}^n)^* \;:\; \mathbb{T}_{x}\overline{X} \subseteq H, \, H \in L_i^\vee \right\rbrace.
\end{align*}
The set $\mathcal{N}_i$ is the intersection of the \emph{conormal variety}
$\mathcal{N}_{\overline{X},\overline{X}^\vee}:= \{ (x,H) \in \overline{X} \times (\mathbb{P}^n)^* : \mathbb{T}_{x}\overline{X} \subseteq H\}$
with the line $L_i^\vee \subseteq (\mathbb{P}^n)^*$ (more formally, with $\mathbb{P}^n \times L_i^\vee$).
Note that the dual hypersurface $\overline{X}^\vee$ is the image of the projection of the conormal variety onto the second factor.
Hence, if the line $L_i^\vee$ was generic, then the set $\mathcal{N}_i$ would be finite of cardinality $\deg  \overline{X}^\vee$.
Thus, for a sufficiently generic polyhedral distance $D$, we would be done, since Corollary \ref{facet} implies that the number of full-dimensional Voronoi cells is at most $\sum_{i=1}^k |\mathcal{N}_i| = k \cdot \deg  \overline{X}^\vee$.

For  an arbitrary polyhedral distance $D$, the varieties $\mathcal{N}_i$ might be infinite.
We write $\mathcal{N}_i = \mathcal{Z}_i \cup \mathcal{N}'_i$, where $\mathcal{Z}_i$ is the positive-dimensional components of $\mathcal{N}_i$ and $\mathcal{N}'_i$ contains the remaining points of $\mathcal{N}_i$.
If $\mathcal{Z}_i$ is empty, then $\mathcal{N}_i = \mathcal{N}'_i$ contains $ \deg  \overline{X}^\vee$ many  points, counted with multiplicity.
Otherwise, there are less than $ \deg  \overline{X}^\vee$ many  points in $\mathcal{N}'_i$.
The latter was shown for the intersection of a subvariety of $\mathbb{P}^m$ with a subspace of complementary dimension in \cite[Proposition 2.1]{jiang2021linear}, but the proof works the same for subvarieties of a product of projective spaces.
Hence, in either case, we conclude that $|\mathcal{N}'_i| \leq \deg  \overline{X}^\vee$.

In the remainder of this proof, we will show that the positive-dimensional components $\mathcal{Z}_i$ of $\mathcal{N}_i$ do not contribute to full-dimensional Voronoi cells, i.e., 
that a full-dimensional Voronoi cell at $x \in X$ implies that $(x,H) \in \mathcal{N}'_i$ for some hyperplane $H$.
This concludes the proof, as it implies that the number of full-dimensional Voronoi cells is at most $\sum_{i=1}^k |\mathcal{N}'_i| \leq k \cdot \deg  \overline{X}^\vee$.

First, we consider irreducible components of $\mathcal{Z}_i$ where (at least) one of the hyperplanes $H$ is tangent at infinitely many points $x$ of $X$.
As in second paragraph of the proof of Theorem \ref{thm:noFullCell}, the tangent hyperplane $H$ does not cause a full-dimensional Voronoi cell at any such point $x$, 
since any $D$-ball with a facet $F$ that is contained in $H$ and has $x$ in its relative interior also has infinitely many other points on its boundary.

After removing all such components from $\mathcal{Z}_i$, the only remaining components in $\mathcal{Z}_i$ (if any) are curves:
The projection of each such curve $\Gamma \subseteq \mathcal{Z}_i \subseteq \overline{X} \times L_i^\vee$ onto the second factor is the whole line $L_i^\vee$, and 
for every hyperplane $H \in L_i$ there are finitely many points $x \in \overline{X}$ such that $(x,H) \in \Gamma$.
The projection of $\Gamma$ onto the first factor is a curve $C \subseteq \overline{X}$.
Every tangent line of $C$ is contained in one of the hyperplanes $H \in L_i$, which means that the tangent line intersects the projective subspace $L_i \subseteq \mathbb{P}^n$ of codimension two.
By Lemma \ref{lem:degenerateCurve}, the curve $C$ must be contained in one of the hyperplanes $H \in L_i$.
Since there are only finitely many points $x \in C$ with $(x,H) \in \Gamma$, for each of the remaining points $x' \in C$ there must be another hyperplane $H' \in L_i$ such that  $(x',H') \in \Gamma$.
In particular, we obtain that $x' \in H \cap H' = L_i$.
This shows that the whole curve $C$ is in fact contained in $L_i$.
Hence, the curve $C$ lies at infinity and not in the affine ambient space $\mathbb{A}^n$, but we are only interested in Voronoi cells of the affine variety $X\subseteq \mathbb{A}^n$.
 \end{proof}

If we restrict ourselves to specific polyhedral distances, the bound in \Cref{polarmax} can be specialized accordingly. For instance,
in an $n$-dimensional metric space, the number of $(m-1)$-dimensional faces of a generic Wasserstein ball is 
 $
 \frac{(n+m)!}{(m!)^2(n-m)!}
 $ \cite{GP}.
 In particular, a generic Wasserstein ball has $\binom{2n}{n}$ many facets.
 In general, for arbitrary Wasserstein balls $B_{W_d}$, that number is an upper bound for the number of facets, i.e., $F(B_{W_d})\leq\binom{2n}{n}$ \cite{GP}. This yields the following result.
\begin{corollary}
    Let $X\subseteq \mathbf{1}_n$ be a smooth irreducible variety such that the dual variety $\overline{X}^\vee$ of its complex projective closure $\overline{X}$ is a hypersurface in $(\mathbb{P}^n)^*$. 
    The number of full-dimensional Voronoi cells of $X$ under a Wasserstein distance is at most
    \[
       \binom{2n}{n}\frac{ \deg(\overline{X}^\vee)}{2}.
    \]
\end{corollary}

\begin{remark}
% More generally, a face $f$ of a generic Wasserstein ball $B_D$ can also produce $(f+1)$-dimensional Voronoi cells and Theorem \ref{transverse} shows that non-transversal intersection is needed for this to happen. This means that the dimension of the space spanned by the tangent space of $X$ at some point $x$ and a parallel translate of the span of $f$ is less than the dimension $n$ of the ambient space. Comparing this with the definition of the polar varieties (\ref{polarvariety}), we see that $i=k-n+\text{dim}(f)+1$ is sufficent. More precisely, the parallel translates of $\mathrm{span}(f)$ in the affine ambient space correspond to projective subspaces of dimension dim$(f)$ that contain a common subspace $V$ of dimension dim$(f)-1$ that lies in the hyperplane at infinity. Hence, if non-transversal intersection happens at a point $x \in X \subset \overline{X}$, then $x$ must be in $P_i(\overline{X},V)$. In other words, the points $x \in X$ such that the face-cone component of the Voronoi cell $V_{D,X}(x)$ corresponding to $f$ is not empty are all contained in the $i$'th polar variety of $\overline{X}$ with respect to $V$.
 %
 In general, we expect infinitely many lower-dimensional Voronoi cells (i.e., of dimension smaller than $n$), so we cannot hope to count them with algebraic invariants, e.g., with \emph{polar degrees} that generalize $\deg \overline{X}^\vee$. On a related matter, polar degrees provide an upper bound for the number of critical points of computing the Wasserstein distance from a point to the variety \cite[Theorem 13 and Proposition 17]{CJMSV}. 
 Figure \ref{fig:3.4} displays a critical point of this kind.
  \end{remark}

We conclude with an example where the bound in Theorem \ref{polarmax} is tight. 
In general, the quantity \eqref{eq:upperBoundCounting} is only an upper bound since it counts how often a facet is tangent to the variety, which is a necessary but not sufficient condition for a full-dimensional Voronoi cell (cf. Corollary \ref{total}).
  In addition, going from the real affine variety $X$
 to its  complex projective closure $\overline{X}$ can introduce additional tangency points.

 \begin{example}
 Fix any polyhedral distance and consider the $n$-sphere in an $(n+1)$-dimensional ambient space. 
 It has another $n$-sphere as its dual variety which has degree two.
 This means that the upper bound \eqref{eq:upperBoundCounting} on the number of full-dimensional Voronoi cells equals the number of facets $F(B_D)$ of the $D$-ball $B_D$.
 
 We also know that for every hyperplane there are exactly two 
 parallel translates that are tangent to the $n$-sphere. 
 For any such tangent hyperplane,  one of its two closed halfspaces intersects the $n$-sphere only at the point of tangency.
 Therefore, we can conclude from Corollary \ref{total}$(c)$ that every facet of the $D$-ball contributes one full-dimensional Voronoi cell. 
 Hence, there are exactly $F(B_D)$ many full-dimensional Voronoi cells, as predicted by the upper bound \eqref{eq:upperBoundCounting} in Theorem \ref{polarmax}.
 \end{example}

 \bigskip  \bigskip  \bigskip
\paragraph{\textbf{Acknowledgements}}
Kathlén Kohn was partially supported by the Knut and Alice Wallenberg Foundation within their WASP (Wallenberg AI, Autonomous Systems and Software Program) AI/Math initiative.
Lorenzo Venturello was supported by the G\"oran Gustafsson foundation.

\bibliographystyle{alpha} 
\bibliography{bibliography.bib}

\end{document}